\documentclass{siamart251216}

\usepackage{amssymb,mathtools}
\usepackage{bm}
\usepackage{enumitem}
\usepackage{microtype}
\usepackage{tikz}
\usetikzlibrary{patterns,arrows.meta}

\newsiamremark{remark}{Remark}

\newcommand{\dd}{\mathrm{d}}

\newcommand{\distT}{\operatorname{dist}_{\mathcal T_h}}
\newcommand{\supp}{\operatorname{supp}}

\headers{Conforming error estimators cannot be robust}{Y. Li}
\title{Locally Computable Error Estimators for Conforming Approximations of Interface Problems Cannot Be Robust}
\author{Yuwen Li\thanks{School of Mathematical Sciences, Zhejiang University,
866 Yuhangtang Road, Hangzhou 310058, Zhejiang, People's Republic of China
(\email{liyuwen@zju.edu.cn}).}}

\ifpdf
\hypersetup{
  pdftitle={Locally Computable Error Estimators for Conforming Approximations of Interface Problems Cannot Be Robust},
  pdfauthor={Yuwen Li}
}
\fi

\begin{document}
\maketitle

\begin{abstract}
We prove an impossibility result for finite-range locally computable a posteriori error estimators for the conforming finite element discretization of an elliptic interface problem. For a class of interface problems with a checkerboard cross-point and $\{1,M\}$-valued coefficients, we construct two problem instances on the same interface-fitted mesh. The two instances share a piecewise-constant load for which the finite element solution and the data oscillation both vanish. The ratio of their exact energy errors grows at least proportionally to $M^{1/4}$. Locality together with efficiency forces identical estimator values for these instances. The product of reliability and efficiency constants is therefore bounded below by a constant multiple of $M^{1/4}$. Consequently, any locally computable error estimator for interface problems, whether of residual, equilibrated, or recovery type, cannot be simultaneously reliable and efficient with contrast-independent constants.
\end{abstract}



\section{Introduction}\label{sec:introduction}

Elliptic interface problems model diffusion in composite media, including heat
conduction and porous-media flow, where material coefficients may jump sharply
across internal interfaces.
A priori error analysis for discontinuous coefficients began with the analysis of
interface approximation and mesh mismatch and later produced nearly optimal
energy- and $L^2$-error estimates for fitted elliptic and parabolic interface
problems~\cite{Babuska1970,ChenZou1998}.  For fitted $P_1$ elements, an
$O(h|\log h|^{1/2})$ broken-$H^1$ estimate for piecewise $H^2$ solutions was
followed by optimal conforming and nonconforming results on maximal-angle
meshes~\cite{Xu1982,XuZhang2016}.  Subsequent work tracked the coefficient
contrast explicitly: diffusion-weighted DG schemes give robust estimates for
smooth as well as low-regularity solutions, and locally optimal CR and DG
estimates are available in two and three dimensions without assumptions on the
coefficient distribution~\cite{ErnStephansenZunino2009,CaiYeZhang2011,
CaiHeZhang2017DG}.  On unfitted meshes, optimal weighted-$H^1$, $L^2$, and flux
estimates independent of the contrast have also been established
\cite{GuzmanSanchezSarkis2017,BurmanGuzmanSanchezSarkis2018}.

A posteriori error estimation grew out of the effort to turn a computed finite
element solution into a quantitative assessment of its own accuracy and, in
turn, into a reliable driver for adaptive refinement.  Popular error estimators include residual-type
estimators, equilibrated residual error 
estimators, and recovery estimators. 
At the a posteriori level, elliptic interface problems add a second
requirement.  If the diffusion coefficient has large jumps, the constants in ideal reliability and efficiency bounds should not deteriorate with the
contrast.
Bernardi and Verf\"urth~\cite{BernardiVerfurth2000} gave a pioneering a posteriori error
analysis for conforming finite elements with discontinuous coefficients. Their residual-type error estimator is coefficient-weighted and is robust for interface jumps when the coefficient distribution satisfies a \emph{quasi-monotonicity}
condition.  
Petzoldt~\cite{Petzoldt2002} refined this condition and established robust residual
estimates in three dimensions for quasi-monotone coefficient
distributions. 

At the approximation-theoretic level, Tantardini and
Verf\"urth~\cite{TantardiniVerfurth2021} showed that, without
quasi-monotonicity, the conforming best-approximation error cannot in general
be robustly localized on elements, pairs of elements, or vertex stars; their
result does not concern computable a posteriori estimators. For mixed finite elements, Vohral\'\i k~\cite{Vohralik2010} proposed guaranteed equilibrated estimators that are uniformly locally efficient under quasi-monotonicity;
without it, the estimator involves a global potential
reconstruction. For conforming $P_1$ discretizations, his robust error estimator for arbitrary coefficient patterns requires a nonlocal dual
norm~\cite{Vohralik2011}. In addition,
Cai's thesis~\cite{Cai2019Thesis} exhibited the loss of contrast-independent local efficiency for
a particular equilibrated estimator based on patchwise $RT_0$ minimization when
quasi-monotonicity fails. However, this estimator-specific counterexample does
not exclude the possibility of other robust and locally computable estimators.

The limitation of conforming methods motivated the study of nonconforming methods for interface problems. An equilibrated estimator for nonconforming elements explicitly treated large coefficient jumps, although its error-bound constants retained a dependence on the contrast~\cite{Ainsworth2005}. For the Crouzeix--Raviart (CR) method, the modified residual
estimator in~\cite{CaiHeZhang2017CR} combines element residuals and normal-flux
jumps with an edge jump contribution that is modified near vertices
where quasi-monotonicity fails. The a posteriori error estimate in \cite{CaiHeZhang2017CR} is therefore robust with respect to coefficient
jumps without a quasi-monotonicity assumption. A subsequent unified treatment of CR
and interior-penalty discontinuous Galerkin methods removed this modification~\cite{CaiHeZhang2017DG} and still achieved unconditionally robust a posteriori error analysis for interface problems.

These positive results do not answer the long-standing question: is there a robust and locally computable a posteriori error estimator for the conforming method for interface problems? In the present work, we provide a negative answer to this question. We achieve this by considering the conforming $P_1$ method at
a non-quasi-monotone checkerboard vertex with the alternating coefficient pattern $1,M,1,M$. We formalize local computability by imposing a fixed information range of $L\in\mathbb{N}_0$ mesh layers; see Section~\ref{sec:locality}. We construct two problem instances that are identical on a local region but have energy-error scales separated
by a factor of order $M^{1/4}$.  The obstruction requires no special structure of the estimator. For example, it rules out robust explicit residual-type estimators and implicit estimators based on local patch solves, such as equilibrated-residual and recovery-type estimators.

\subsection{Problem setting}

Let $\Omega=(-1,1)^2$
and denote its four quadrants by
\[
\begin{aligned}
 Q_1&=(0,1)\times(0,1),
 &Q_2&=(-1,0)\times(0,1),\\
 Q_3&=(-1,0)\times(-1,0),
 &Q_4&=(0,1)\times(-1,0).
\end{aligned}
\]
For a piecewise-constant coefficient $\alpha>0$ aligned with $\{Q_1, Q_2, Q_3, Q_4\}$ and a datum $f\in L^2(\Omega)$, let the exact solution $u\in H_0^1(\Omega)$ solve
\begin{equation}\label{eq:pde}
  a_\alpha(u,v):=\int_\Omega \alpha\nabla u\cdot\nabla v\,\dd x
  =(f,v)_{L^2(\Omega)}
  \qquad\forall v\in H_0^1(\Omega).
\end{equation}
The contrast of a piecewise-constant coefficient $\alpha$ is
\[
M(\alpha)=\frac{\max_{K\in\mathcal{T}_h}\alpha|_K}{\min_{K\in\mathcal{T}_h}\alpha|_K}.
\]
For example, we shall consider the checkerboard diffusion coefficient
\begin{equation}\label{eq:alphaA}
  \alpha=\alpha^A(x):=
  \begin{cases}
     1,&x\in Q_1\cup Q_3,\\
     M,&x\in Q_2\cup Q_4,
  \end{cases}
  \qquad M>1.
\end{equation}
Here the contrast is $M(\alpha^A)=M$.
The inner boundaries of $Q_1, Q_2, Q_3, Q_4$ are called interfaces, and $\alpha^A$ is an interface-fitted coefficient.

Let $\mathcal{T}_h$ be an interface-fitted conforming shape-regular triangulation and
\[
  V_h=\{v_h\in H_0^1(\Omega):v_h|_K\in P_1(K)\ \forall K\in\mathcal{T}_h\}.
\]
Here $P_1(K)$ is the space of affine polynomials in two variables restricted to $K$.
The conforming $P_1$ finite element solution $u_h\in V_h$ is defined by
\begin{equation}\label{eq:P1FEM}
  a_\alpha(u_h,v_h)=(f,v_h)_{L^2(\Omega)}
  \qquad\forall v_h\in V_h.
\end{equation}
We shall consider an a posteriori error estimate in the energy norm $\|v\|_\alpha^2=a_\alpha(v,v)$.

The rest of the paper is organized as follows.  Section~\ref{sec:locality}
formulates finite-range local computability, records the reliability and
efficiency requirements, and constructs a piecewise-constant load that is
invisible to the conforming finite element space.  Section~\ref{sec:instances}
builds two problem instances and derives their contrasting energy
bounds.  Section~\ref{sec:nogo} combines these estimates to prove the no-go
theorem and its consequence.

\section{Finite-range local computability}\label{sec:locality}

For $K,T\in\mathcal{T}_h$, let $\distT(K,T)$ be the graph distance in the element face-neighbor graph.  For a set of elements $\mathcal S\subseteq\mathcal{T}_h$, write
\[
  \distT(K,\mathcal S)=\min_{T\in\mathcal S}\distT(K,T)
\]
and define the $L$-layer neighborhood
\begin{equation*}
  \mathcal N_L(\mathcal S)
  :=\{K\in\mathcal{T}_h:\distT(K,\mathcal S)\le L\}.
\end{equation*}
For a single element $K$, the associated geometric patch is
\[
  \omega_L(K):=\operatorname{int}\!\left(\bigcup_{T\in\mathcal{T}_h:\,\distT(T,K)\le L}\overline T\right).
\]

A problem instance will be denoted by 
\[
  \mathcal P=(\mathcal{T}_h,\alpha,f,u_h),
\]
where $u_h$ is the $P_1$ solution in \eqref{eq:P1FEM} corresponding to the triple $(\mathcal{T}_h,\alpha,f)$. We say that the pair $(\mathcal{T}_h,\alpha)$ is admissible with contrast $M>1$ if $\mathcal{T}_h$ is a conforming triangulation and $\alpha$ is piecewise constant on $\mathcal{T}_h$, satisfies $\alpha|_K\in\{1,M\}$ for every $K\in\mathcal{T}_h$, and attains both values $1$ and $M$. A problem $\mathcal{P}=(\mathcal{T}_h,\alpha,f,u_h)$ is admissible provided $(\mathcal{T}_h,\alpha)$ is an admissible pair and $f\in L^2(\Omega)$.

For an open union $\omega$ of mesh elements, we write
\[
  \mathcal P|_\omega
  :=\bigl(\mathcal T_h|_\omega,\alpha|_\omega,
          f|_\omega,u_h|_\omega\bigr),
  \qquad
  \mathcal T_h|_\omega
  :=\{T\in\mathcal T_h:\operatorname{int}(T)\subset\omega\}.
\]
Thus, two problem instances agree on $\omega$ if their restricted meshes agree
and their coefficients, loads, and finite element solutions agree almost
everywhere on $\omega$.

An a posteriori estimator for $\mathcal{P}$ is of the form 
\[
  \eta(\mathcal P)^2=\sum_{K\in\mathcal{T}_h}\eta_K(\mathcal P)^2,
  \qquad \eta_K(\mathcal P)\ge0.
\] 
Each nonnegative error indicator $\eta_K(\mathcal P)$ is intended to quantify the contribution from triangle $K$ to the energy-norm error.

\begin{definition}[Finite-range locally computable estimator]\label{def:local-computable}
Fix a range $L\in\mathbb N_0$.  An a posteriori estimator $\eta$ is called locally computable with range $L$ if for any two admissible problems $\mathcal P=(\mathcal{T}_h,\alpha,f,u_h)$ and $\widetilde{\mathcal P}
  =(\mathcal{T}_h,\widetilde{\alpha},\widetilde f,\widetilde u_h)$ on the same mesh and for any $K\in\mathcal{T}_h$,
\[
\mathcal{P}|_{\omega_L(K)}=\widetilde{\mathcal{P}}|_{\omega_L(K)}\Longrightarrow\eta_K(\mathcal P)=\eta_K(\widetilde{\mathcal P}).
\]
\end{definition}
No special structure of the error indicator, such as an explicit residual-type estimator or an implicit estimator defined by a particular local solve, is assumed in Definition~\ref{def:local-computable}. Common residual-type, equilibrated residual, and recovery-type error estimators satisfy this definition. 

\begin{remark}\label{rem:local-computable}
The finite-range local computability definition makes precise what is meant here by a locally defined estimator and excludes hidden global information. If only one
fixed checkerboard problem were considered, an estimator could be tailored in
advance to its complete global structure, for example, by storing data derived
from the inverse of global operators. The resulting indicator might
be assigned to a single element even though its construction used information
from the entire domain. Definition~\ref{def:local-computable} rules this out by
requiring the local indicators to agree whenever two admissible problems agree
on the corresponding $L$-layer patch.
\end{remark}

Let $P_0(\mathcal{T}_h)$ denote the space of piecewise-constant functions on $\mathcal{T}_h$.
Let $\operatorname{osc}_h(\alpha,f)\ge0$ denote any data-oscillation functional satisfying
\begin{equation}\label{eq:osc-consistency}
  f\in P_0(\mathcal{T}_h)
  \quad\Longrightarrow\quad
  \operatorname{osc}_h(\alpha,f)=0.
\end{equation}
This includes the standard elementwise weighted oscillation
\cite{BernardiVerfurth2000,CaiHeZhang2017DG,CaiHeZhang2017CR,Petzoldt2002,Vohralik2010,Vohralik2011} for interface problems
\begin{equation*}
  \operatorname{osc}_h(\alpha,f)^2
  :=\sum_{K\in\mathcal{T}_h}h_K^2\alpha_K^{-1}
       \|f-\Pi_K f\|_{L^2(K)}^2,
\end{equation*}
where $\Pi_K$ is the $L^2(K)$-orthogonal projection onto constants. Our theory does not assume any specific form of data oscillation beyond the consistency condition in \eqref{eq:osc-consistency}. The load constructed below is piecewise constant, so its oscillation vanishes.

We consider an a posteriori error estimator $\eta$ intended to satisfy the two-sided energy-error bounds for \emph{all} admissible problems $\mathcal{P}=(\mathcal{T}_h,\alpha,f,u_h)$:
\begin{align}
  \|u-u_h\|_\alpha
  &\le C_{\rm rel}(M)
      \bigl(\eta(\mathcal{P})+\operatorname{osc}_h(\alpha,f)\bigr),
  \label{eq:rel}\\
  \eta(\mathcal{P})
  &\le C_{\rm eff}(M)
      \bigl(\|u-u_h\|_\alpha+\operatorname{osc}_h(\alpha,f)\bigr).
  \label{eq:eff}
\end{align}
In the literature, \eqref{eq:rel} and \eqref{eq:eff} are called reliability and efficiency of $\eta$, respectively.
A jump-robust estimator would have both constants $C_{\rm rel}(M)>0$ and $C_{\rm eff}(M)>0$ bounded independently of $M$. If $\eta$ fails to be robust for the checkerboard interface problem with admissible $\alpha$, it has no chance to be robust for more general problems.

\begin{lemma}\label{lem:null-patch}
Assume that an estimator locally computable with range $L$ satisfies the efficiency bound \eqref{eq:eff}.  Let $\mathcal P=(\mathcal{T}_h,\alpha,f,u_h)$ and $K\in\mathcal{T}_h$. Then
\[
f|_{\omega_L(K)}=0\quad\text{and}\quad u_h|_{\omega_L(K)}=0\Longrightarrow\eta_K(\mathcal P)=0.
\] 
\end{lemma}

\begin{proof}
Let $\mathcal P_0=(\mathcal{T}_h,\alpha,0,0)$ be the zero-load problem instance with $u_h=u=0$.  The efficiency bound \eqref{eq:eff} then forces
\[
  \eta(\mathcal P_0)=0.
\]
By the nonnegativity of $\eta$, we have $\eta_T(\mathcal P_0)=0$ for every $T\in\mathcal{T}_h$.  On $\omega_L(K)$, the two problems $\mathcal P$ and $\mathcal P_0$ coincide, so Definition~\ref{def:local-computable} yields
\[
  \eta_K(\mathcal P)=\eta_K(\mathcal P_0)=0.
\]
The proof is complete. 
\end{proof}

\subsection{A piecewise-constant invisible load}\label{sec:load}

Fix a range $L\in\mathbb N_0$ of local layers.  We shall construct one mesh $\mathcal{T}_h$, independent of $M$, and a datum in $P_0(\mathcal{T}_h)$ that is invisible to the conforming $P_1$ finite element space but has nonzero and opposite total masses in the two coefficient-$M$ branches.

Choose fixed numbers
\[
  0<r_0<r_1<r_2<1.
\]
Define the two congruent macrotriangles
\begin{align*}
 T_+&=\operatorname{conv}\{O,A_+,B_+\}\subset Q_2,
 \\
 T_-&=-T_+=\operatorname{conv}\{O,A_-,B_-\}\subset Q_4,
 \\
 O&=(0,0),\quad A_+=(-r_0,0),\quad B_+=(0,r_0),\\
 A_-&=(r_0,0),\quad B_-=(0,-r_0).
\end{align*}
Introduce the interior points
\begin{align*}
 C_\pm&=\frac{A_\pm+B_\pm}{3},
 &E_\pm&=\frac12A_\pm+\frac1{12}B_\pm,
 &D_\pm&=\frac1{12}A_\pm+\frac12B_\pm,
\end{align*}
and subdivide each macrotriangle into
\begin{align*}
 K_1^\pm&=\operatorname{conv}\{A_\pm,B_\pm,C_\pm\},
 &K_2^\pm&=\operatorname{conv}\{O,A_\pm,E_\pm\},\\
 K_3^\pm&=\operatorname{conv}\{O,E_\pm,C_\pm\},
 &K_4^\pm&=\operatorname{conv}\{A_\pm,C_\pm,E_\pm\},\\
 K_5^\pm&=\operatorname{conv}\{O,B_\pm,D_\pm\},
 &K_6^\pm&=\operatorname{conv}\{O,D_\pm,C_\pm\},\\
 K_7^\pm&=\operatorname{conv}\{B_\pm,C_\pm,D_\pm\}.&&
\end{align*}
Figure~\ref{fig:macro-subdivision} shows the subdivision of $T_+$; the subdivision of $T_-$ is obtained by reflection through $O$.

\begin{figure}[htbp]
\centering
\begin{tikzpicture}[scale=1.38,
  point/.style={circle,fill=black,inner sep=1.35pt},
  cell/.style={font=\scriptsize,fill=white,inner sep=0.6pt}]
  \coordinate (O) at (0,0);
  \coordinate (A) at (-4,0);
  \coordinate (B) at (0,4);
  \coordinate (C) at (-1.3333,1.3333);
  \coordinate (E) at (-2,0.3333);
  \coordinate (D) at (-0.3333,2);

  \filldraw[fill=gray!14,thick] (A)--(B)--(C)--cycle;
  \filldraw[fill=gray!28,thick] (O)--(A)--(E)--cycle;
  \filldraw[fill=gray!8,thick]  (O)--(E)--(C)--cycle;
  \filldraw[fill=gray!22,thick] (A)--(C)--(E)--cycle;
  \filldraw[fill=gray!28,thick] (O)--(B)--(D)--cycle;
  \filldraw[fill=gray!8,thick]  (O)--(D)--(C)--cycle;
  \filldraw[fill=gray!22,thick] (B)--(C)--(D)--cycle;

  \node at (-1.78,1.82) {$K_1^+$};
  \node at (-1.05,0.58) {$K_3^+$};
  \node at (-2.3,0.59) {$K_4^+$};
  \node at (-0.58,1.08) {$K_6^+$};
  \node at (-0.5,2.4) {$K_7^+$};
  \node (K2label) at (-2.05,-0.48) {$K_2^+$};
  \draw[-{Latex[length=1.3mm]},thin] (K2label.north) -- (-2.0,0.12);
  \node (K5label) at (0.72,2.30) {$K_5^+$};
  \draw[-{Latex[length=1.3mm]},thin] (K5label.west) -- (-0.09,2.0);

  \node[point] at (O) {};
  \node[point] at (A) {};
  \node[point] at (B) {};
  \node[point] at (C) {};
  \node[point] at (E) {};
  \node[point] at (D) {};
  \node[font=\small] at (0.18,-0.18) {$O$};
  \node[font=\small] at (-4.18,-0.18) {$A_+$};
  \node[font=\small] at (0.14,4.18) {$B_+$};
  \node[font=\scriptsize,inner sep=0.5pt] (Clabel) at (-1.7,1.4) {$C_+$};
  \draw[-{Latex[length=1.1mm]},thin] (Clabel.east) -- (C);
  \node[font=\scriptsize,inner sep=0.5pt] (Elabel) at (-2.3,0.15) {$E_+$};
  \draw[-{Latex[length=1.1mm]},thin] (Elabel.east) -- (E);
  \node[font=\scriptsize,fill=white,inner sep=0.5pt] (Dlabel) at (0.32,1.62) {$D_+$};
  \draw[-{Latex[length=1.1mm]},thin] (Dlabel.west) -- (D);
\end{tikzpicture}
\caption{The seven-triangle subdivision of the macrotriangle $T_+$.  The reflected macrotriangle $T_-=-T_+$ has the corresponding subdivision with superscript $-$.}
\label{fig:macro-subdivision}
\end{figure}
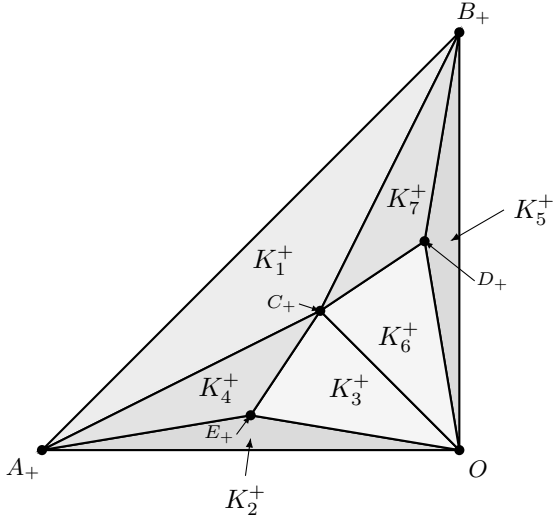

We construct an interface-fitted conforming triangulation $\mathcal{T}_h$ with the following three properties:
\begin{enumerate}[label=(M\arabic*)]
\item the coordinate axes and the squares
\[
  U_1=(-r_1,r_1)^2,\qquad U_2=(-r_2,r_2)^2
\]
are aligned with the mesh $\mathcal{T}_h$;
\item the restriction of $\mathcal{T}_h$ to $T_+$ and $T_-$ is precisely the element set 
\[
  \mathcal S_0:=\big\{K_j^+,K_j^-:\, 1\leq j\leq7\big\};
\]
\item it holds that $\mathcal N_{2L}(\mathcal S_0)
  \subset\{K\in\mathcal{T}_h:\overline K\subset U_1\}$.
\end{enumerate}

Such a mesh exists for every fixed $L$: retain the two fixed macrotriangle subdivisions and insert sufficiently many layers of triangles between them and $\partial U_1$.  The mesh quality of $\mathcal{T}_h$ depends on $L$ but is independent of $M$.

Set $(\mu_1,\ldots,\mu_7)
 =\left(\frac13,\frac16,\frac13,-\frac12,
         \frac16,\frac13,-\frac12\right)$.
Define $f\in P_0(\mathcal{T}_h)$ by
\begin{equation}\label{eq:fdef}
 f|_{K_j^+}:=\frac{3\mu_j}{|K_j^+|},
 \qquad
 f|_{K_j^-}:=-\frac{3\mu_j}{|K_j^-|},
 \qquad j=1,\ldots,7,
\end{equation}
and set $f=0$ on every other element of $\mathcal{T}_h$.
Thus $f\in P_0(\mathcal{T}_h)$ is piecewise constant on $\mathcal{T}_h$, is supported entirely in $Q_2\cup Q_4$, and is independent of $M$. By the consistency of data oscillation \eqref{eq:osc-consistency}, $f$ satisfies
\begin{equation}
 \operatorname{osc}_h(\alpha,f)=0
 \qquad\text{for every admissible }(\mathcal{T}_h,\alpha).
 \label{eq:osc-zero}
\end{equation}

A direct calculation yields the following properties of $f$. 
\begin{lemma}\label{lem:f-properties}
The datum $f$ in \eqref{eq:fdef} satisfies
\begin{subequations}
  \begin{align}
 (f,v_h)_{L^2(\Omega)}&=0,\qquad\forall v_h\in V_h,
 \label{eq:f-orthogonal}\\
 \int_{T_+}f\,\dd x&=1,
 \qquad\int_{T_-}f\,\dd x=-1
 \label{eq:branch-masses}.
\end{align}
\end{subequations}
\end{lemma}
Consequently, the load $f$ is invisible to the conforming $P_1$ method; for every positive $\alpha\in P_0(\mathcal{T}_h)$, the conforming $P_1$ solution in \eqref{eq:P1FEM} is $u_h=0$. The property in \eqref{eq:branch-masses} is used to control the lower bound for the energy error of a problem instance.

\section{Two problem instances}\label{sec:instances}

We compare two coefficient fields on the same domain and the same mesh. Let $\alpha^A$ be exactly the four-piece coefficient \eqref{eq:alphaA}. Both coefficient-$M$ quadrants $Q_2$ and $Q_4$ extend from the cross-point to the Dirichlet boundary.

Next, we consider a coefficient with truncated coefficient-$M$ branches:
\begin{equation}\label{eq:alphaB}
  \alpha^B(x)=
  \begin{cases}
    \alpha^A(x),&x\in U_1=(-r_1,r_1)^2,\\
    1,&x\in\Omega\setminus U_1.
  \end{cases}
\end{equation}
Because $\partial U_1$ is mesh-aligned, $\alpha^B$ is piecewise constant on $\mathcal{T}_h$.  The two coefficient fields coincide throughout $U_1$, but the two $M$-branches of $\alpha^B$ terminate at $\partial U_1$ rather than reaching the physical boundary; see Figure~\ref{fig:coefficient}.

Let $u^A,u^B$ be the exact solutions in \eqref{eq:pde} for $\alpha=\alpha^A$ and $\alpha=\alpha^B$, respectively.  By Lemma~\ref{lem:f-properties}, the finite element approximations of $u^A$ and $u^B$ both vanish:
\begin{equation}\label{eq:uhzero}
  u_h^A=u_h^B=0.
\end{equation}
We define the two problem instances
\begin{equation}\label{eq:PAPB}
    \mathcal P^A:=(\mathcal{T}_h,\alpha^A,f,0),\qquad \mathcal P^B:=(\mathcal{T}_h,\alpha^B,f,0).
\end{equation}
Here $\mathcal P^A$ and $\mathcal P^B$ differ only in the far continuation of the coefficient. The contrast between $\alpha^A$ and $\alpha^B$ comes from whether the two coefficient-$M$ branches can discharge independently at the Dirichlet boundary.

\begin{figure}[htbp]
\centering
\begin{tikzpicture}[scale=2.45]
  \begin{scope}[xshift=-1.25cm]
    \draw[thick] (-1,-1) rectangle (1,1);
    \fill[gray!35] (-1,0) rectangle (0,1);
    \fill[gray!35] (0,-1) rectangle (1,0);
    \draw[thick] (-1,0)--(1,0);
    \draw[thick] (0,-1)--(0,1);
    \fill (0,0) circle (0.025);
    \node at (-0.5,0.55) {$M$};
    \node at (0.5,0.55) {$1$};
    \node at (-0.5,-0.55) {$1$};
    \node at (0.5,-0.55) {$M$};
    \draw[very thick] (-0.32,0)--(0,0.32)--(0,0)--cycle;
    \draw[very thick] (0.32,0)--(0,-0.32)--(0,0)--cycle;
    \node[font=\scriptsize] at (-0.105,0.07) {$T_+$};
    \node[font=\scriptsize] at (0.105,-0.105) {$T_-$};
  \end{scope}
  \begin{scope}[xshift=1.25cm]
    \draw[thick] (-1,-1) rectangle (1,1);
    \draw[dashed] (-0.5,-0.5) rectangle (0.5,0.5);
    \fill[gray!35] (-0.5,0) rectangle (0,0.5);
    \fill[gray!35] (0,-0.5) rectangle (0.5,0);
    \draw[thick] (-0.5,0)--(0.5,0);
    \draw[thick] (0,-0.5)--(0,0.5);
    \fill (0,0) circle (0.025);
    \node at (-0.25,0.29) {$M$};
    \node at (0.25,0.29) {$1$};
    \node at (-0.25,-0.29) {$1$};
    \node at (0.25,-0.29) {$M$};
    \node at (0,0.75) {$1$};
    \draw[very thick] (-0.32,0)--(0,0.32)--(0,0)--cycle;
    \draw[very thick] (0.32,0)--(0,-0.32)--(0,0)--cycle;
    \node[font=\scriptsize] at (-0.105,0.07) {$T_+$};
    \node[font=\scriptsize] at (0.105,-0.105) {$T_-$};
  \end{scope}
\end{tikzpicture}
\caption{Distributions of $\alpha^A$ (left) and $\alpha^B$ (right).}
\label{fig:coefficient}
\end{figure}
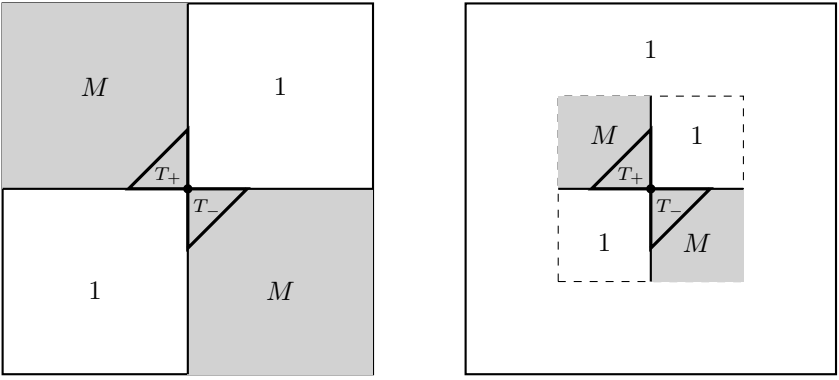

\subsection{Problem A}
The next lemma gives an upper bound on the energy of the exact solution $u^A$ in terms of $M$.
\begin{lemma}\label{lem:A-upper}
With $C_A:=\sqrt{639/5}\,r_0^{-1}$, one has
\begin{equation*}
  \|u^A\|_{\alpha^A}\le C_A M^{-1/2}.
\end{equation*}
\end{lemma}

\begin{proof}
For $(x_1,x_2)\in Q_2\cup Q_4$, define the flux 
\begin{equation*}
  \bm{\tau}(x_1,x_2):=
  \left(0,\int_0^{x_2} f(x_1,s)\,\dd s\right),
\end{equation*}
and set $\bm{\tau}=\bm{0}$ on $Q_1\cup Q_3$.  Hence
$\nabla\cdot\bm{\tau}=f$ in $Q_2\cup Q_4$.  

By definition, the normal trace of $\bm{\tau}$ vanishes on both axes $\{x_1=0\}$ and $\{x_2=0\}$.
Therefore $\bm{\tau}$ belongs to 
$H(\operatorname{div},\Omega)$ and satisfies
\[
  \nabla\cdot\bm{\tau}=f
  \quad\text{in }\Omega.
\]

By the Cauchy--Schwarz inequality, we have
\begin{equation}\label{eq:intf}
  \begin{aligned}
 \int_0^1\left|\int_0^{x_2} f(x_1,s)\,\dd s\right|^2\dd x_2
 \le \frac12\int_0^1 |f(x_1,s)|^2\dd s,\qquad x_1<0,\\
 \int_{-1}^0\left|\int_0^{x_2} f(x_1,s)\,\dd s\right|^2\dd x_2
 \le \frac12\int_{-1}^0 |f(x_1,s)|^2\dd s,\qquad x_1>0.
\end{aligned}
\end{equation}
Moreover, the areas of $K_1^+,\ldots,K_7^+$, divided by $r_0^2$, are
\[
 \frac16,\ \frac1{24},\ \frac5{72},\ \frac1{18},\
 \ \frac1{24},\ \frac5{72},\ \frac1{18}.
\]
Consequently, \eqref{eq:fdef} gives
\[
 \|f\|_{L^2(\Omega)}^2
 =18\sum_{j=1}^7\frac{\mu_j^2}{|K_j^+|}
 =\frac{1278}{5r_0^2}.
\]
Integrating the preceding estimates \eqref{eq:intf} for $f$ with respect to $x_1$
therefore yields
\begin{equation*}
  \|\bm{\tau}\|_{L^2(\Omega)}
  \le \frac1{\sqrt2}\|f\|_{L^2(\Omega)}
  =\frac1{r_0}\sqrt{\frac{639}{5}}=C_A.
\end{equation*}
Since $\bm{\tau}$ is supported entirely in $Q_2\cup Q_4$, where $\alpha^A=M$,
\begin{equation*}
  \|(\alpha^A)^{-1/2}\bm{\tau}\|_{L^2(\Omega)}\le C_A M^{-1/2}.
\end{equation*}
It then follows from \eqref{eq:pde} and $\nabla\cdot\bm{\tau}=f$ that 
\begin{align*}
  \|u^A\|_{\alpha^A}&=\sup_{v\in H_0^1(\Omega),\,\|v\|_{\alpha^A}=1} \big|(f,v)_{L^2(\Omega)}\big|\\
  &=\sup_{v\in H_0^1(\Omega),\,\|v\|_{\alpha^A}=1} \big|(\bm{\tau},\nabla v)_{L^2(\Omega)}\big|\le C_A M^{-1/2}.  
\end{align*}
The proof is complete. 
\end{proof}

\subsection{Problem B}
In this section, we develop a lower bound on the energy of the exact solution $u^B$ in terms of $M$.
We use polar coordinates $x=(r\cos\theta,r\sin\theta)$ with respect to the origin.  Define the continuous $2\pi$-periodic angular function
\begin{equation*}
 g(\theta)=
 \begin{cases}
  -1+4\theta/\pi,&0\le\theta\le\pi/2,\\
  1,&\pi/2\le\theta\le\pi,\\
  1-4(\theta-\pi)/\pi,&\pi\le\theta\le3\pi/2,\\
  -1,&3\pi/2\le\theta\le2\pi.
 \end{cases}
\end{equation*}
Then $g=1$ on $Q_2$, $g=-1$ on $Q_4$, and $g'$ is supported only in $Q_1\cup Q_3$.  In particular, in the inner checkerboard region $U_1$,
\begin{align}
 \int_0^{2\pi}\alpha^B(\theta)|g(\theta)|^2\,\dd\theta
 &=\pi\left(M+\frac13\right),
 \label{eq:gweighted}\\
 \int_0^{2\pi}\alpha^B(\theta)|g'(\theta)|^2\,\dd\theta
 &=\frac{16}{\pi}.
 \label{eq:gprimeweighted}
\end{align}

Since $r_0<r_1$, the disk $B_{r_0}(0)$ is contained in $U_1$.  Let
$\varepsilon=r_0e^{-\sqrt M}$ and define the continuous radial function
\begin{equation*}
 A_M(r)=
 \begin{cases}
  0,&0\le r\le\varepsilon,\\
  \dfrac{\log(r/\varepsilon)}{\sqrt M},&\varepsilon<r<r_0,\\
  1,&r\ge r_0.
 \end{cases}
\end{equation*}
Choose a fixed Lipschitz cutoff function $\chi$ such that
\begin{equation*}
  0\le\chi\le1,
  \qquad
  \chi=1\text{ on }U_1,
  \qquad
  \chi=0\text{ outside }U_2.
\end{equation*}
It can be chosen with $\supp\nabla\chi\subset U_2\setminus\overline U_1$, where $\alpha^B=1$.  Define
\begin{equation*}
  v_M(x)=\chi(x)A_M(|x|)g(\theta).
\end{equation*}
The radial cutoff vanishes in a neighborhood of the origin, so $v_M\in H_0^1(\Omega)$.

The function $v_M$ will serve as a test function for deriving a lower energy bound for $u^B$. First, we show that its pairing with the
load remains uniformly positive as $M$ increases.
\begin{lemma}\label{lem:pairing}
There exists $M_0>0$ such that
\begin{equation*}
  (f,v_M)_{L^2(\Omega)}\ge 1
  \qquad\forall M\ge M_0.
\end{equation*}
\end{lemma}

\begin{proof}
On $T_+$ one has $\chi=1$ and $g=1$, while on $T_-$ one has $\chi=1$ and $g=-1$.  By reflection symmetry of the two macrotriangles,
\begin{equation*}
  (f,v_M)_{L^2(\Omega)}=2\int_{T_+}f(x)A_M(|x|)\,\dd x.
\end{equation*}
By \eqref{eq:branch-masses}, $\int_{T_+}f\,\dd x=1$.
Using the areas listed in the proof of Lemma~\ref{lem:A-upper}, we obtain
\begin{equation*}
 \|f\|_{L^\infty(T_+)}
 =\max_{1\le j\le 7}\frac{3|\mu_j|}{|K_j^+|}
 =\frac{27}{r_0^2}.
\end{equation*}
Moreover, since $\log(r_0/\varepsilon)=\sqrt M$,
\begin{align*}
 \int_{T_+}|1-A_M(|x|)|\,\dd x
 &\le \int_{B_{r_0}(0)}|1-A_M(|x|)|\,\dd x\\
 &=\pi\varepsilon^2
   +\frac{2\pi}{\sqrt M}\int_\varepsilon^{r_0}
        r\log\frac{r_0}{r}\,\dd r\\
 &=\pi\varepsilon^2
   +\frac{2\pi}{\sqrt M}
    \left[\frac{r^2}{2}\log\frac{r_0}{r}+\frac{r^2}{4}
    \right]_{\varepsilon}^{r_0}\\
 &=\frac{\pi(r_0^2-\varepsilon^2)}{2\sqrt M}
 \le \frac{\pi r_0^2}{2\sqrt M}.
\end{align*}
Consequently,
\[
 \left|\int_{T_+}fA_M\,\dd x-1\right|
 \le \|f\|_{L^\infty(T_+)}
       \int_{T_+}|1-A_M(|x|)|\,\dd x
 \le \frac{27\pi}{2\sqrt M}.
\]
Thus, for $M\ge M_0:=(27\pi)^2$, the right-hand side is at most $1/2$;
hence $\int_{T_+}fA_M\,\dd x\ge1/2$ and $(f,v_M)\ge1$.
\end{proof}

We next estimate the energy of the test function $v_M$.  
\begin{lemma}\label{lem:vMenergy}
There exists a constant $C_B$, independent of $M$, such that
\begin{equation*}
  \|v_M\|_{\alpha^B}^2\le C_B\sqrt M.
\end{equation*}
\end{lemma}

\begin{proof}
In $B_{r_0}(0)$ one has $\chi=1$, and
\[
  |\nabla(A_M g)|^2
  \le 2|A_M'(r)|^2|g(\theta)|^2
      +2\frac{A_M(r)^2}{r^2}|g'(\theta)|^2.
\]
The radial integral is
\begin{equation*}
  \int_\varepsilon^{r_0}|A_M'(r)|^2r\,\dd r
  =\frac1M\int_\varepsilon^{r_0}\frac{\dd r}{r}
  =\frac1{\sqrt M}.
\end{equation*}
Combining this with \eqref{eq:gweighted} gives
\begin{equation}\label{eq:vM-radial}
 \begin{aligned}  
  &\int_{B_{r_0}(0)}
     \alpha^B|A_M'(r)|^2|g(\theta)|^2\,\dd x\\
  &=\left(\int_\varepsilon^{r_0}|A_M'(r)|^2r\,\dd r\right)
     \left(\int_0^{2\pi}\alpha^B(\theta)|g(\theta)|^2\,\dd\theta\right)\\
     &=\frac{\pi}{\sqrt M}\left(M+\frac13\right)
     \le \frac{4\pi}{3}\sqrt M.
\end{aligned} 
\end{equation}

For the angular part,
\begin{equation*}
 \int_\varepsilon^{r_0}\frac{A_M(r)^2}{r}\,\dd r
 =\frac1M\int_0^{\sqrt M}s^2\,\dd s
 =\frac{\sqrt M}{3}.
\end{equation*}
Since $g'$ is supported in the two
coefficient-one quadrants, \eqref{eq:gprimeweighted} gives
\begin{equation}\label{eq:vM-angular}
  \begin{aligned}
  &\int_{B_{r_0}(0)}
    \alpha^B\frac{A_M(r)^2}{r^2}|g'(\theta)|^2\,\dd x\\
  &=\left(\int_\varepsilon^{r_0}\frac{A_M(r)^2}{r}\,\dd r\right)
    \left(\int_0^{2\pi}\alpha^B(\theta)|g'(\theta)|^2\,\dd\theta\right)
    =\frac{16}{3\pi}\sqrt M.
\end{aligned}
\end{equation}

Since $A_M=1$ on $(\supp\chi)\setminus B_{r_0}(0)$, define the outer-region
contribution 
\begin{equation*}
 C_\chi:=\int_{(\supp\chi)\setminus B_{r_0}(0)}\alpha^B|\nabla(\chi A_M g)|^2\,\dd x=\int_{(\supp\chi)\setminus B_{r_0}(0)}
       \alpha^B|\nabla(\chi g)|^2\,\dd x.
\end{equation*}
This constant is independent of $M$: on $U_1$ the derivative $g'$ is supported in $Q_1\cup Q_3$, where $\alpha^B=1$, while $\nabla\chi$ is supported in $U_2\setminus\overline U_1$, where $\alpha^B=1$ as well.  

Applying
\eqref{eq:vM-radial} and \eqref{eq:vM-angular} to the radial and angular terms
yields
\begin{align*}
  \|v_M\|_{\alpha^B}^2
  &=\int_{B_{r_0}(0)}\alpha^B|\nabla(A_M g)|^2\,\dd x+C_\chi\\
  &\le 2\int_{B_{r_0}(0)}
       \alpha^B|A_M'(r)|^2|g(\theta)|^2\,\dd x\\
  &\quad+2\int_{B_{r_0}(0)}
       \alpha^B\frac{A_M(r)^2}{r^2}|g'(\theta)|^2\,\dd x
     +C_\chi\\
  &\le \frac{2\pi}{\sqrt M}\left(M+\frac13\right)
       +\frac{32}{3\pi}\sqrt M+C_\chi\\
  &=\left(2\pi+\frac{32}{3\pi}\right)\sqrt M
       +\frac{2\pi}{3\sqrt M}+C_\chi
   \le C_B\sqrt M,
\end{align*}
where $C_B$ has been enlarged if necessary.  This proves the result.
\end{proof}

The preceding pairing and energy estimates can now be combined through the
dual characterization of the solution norm to obtain the lower bound for
Problem B.
\begin{lemma}\label{lem:B-lower}
For $M\geq M_0$, we have
\begin{equation*}
  \|u^B\|_{\alpha^B}\ge C_B^{-1/2}M^{-1/4}.
\end{equation*}
\end{lemma}

\begin{proof}
Using Lemmas~\ref{lem:pairing} and~\ref{lem:vMenergy}, we obtain
\begin{equation*}
   \|u^B\|_{\alpha^B}
=\sup_{0\ne v\in H_0^1(\Omega)}
       \frac{(f,v)_{L^2(\Omega)}}{\|v\|_{\alpha^B}}\ge \frac{(f,v_M)_{L^2(\Omega)}}{\|v_M\|_{\alpha^B}}\ge C_B^{-1/2}M^{-1/4}.
\end{equation*}
This completes the proof.
\end{proof}

\section{The no-go theorem}\label{sec:nogo}
In this section, we present the main no-go theorem for a posteriori error estimation for interface problems. 

Although $\mathcal P^A$ and $\mathcal P^B$ in \eqref{eq:PAPB} differ globally, their local data
are indistinguishable on every patch where an indicator can be nonzero, as the
next lemma shows.
\begin{lemma}\label{lem:eta-equal}
Let $\eta$ be a locally computable estimator with range $L$ that satisfies the efficiency estimate \eqref{eq:eff}. Then
\[
\eta_K(\mathcal{P}^A)
  =
  \eta_K(\mathcal{P}^B)\qquad\text{for all }K\in\mathcal{T}_h.
\]
\end{lemma}

\begin{proof}
Fix $K\in\mathcal{T}_h$ and consider the following two cases.

Case 1.  If $\distT(K,\mathcal S_0)>L$, then $f=0$ on the patch
$\omega_L(K)$. Since $u_h^A=u_h^B=0$ on $\omega_L(K)$ (see \eqref{eq:uhzero}),
Lemma~\ref{lem:null-patch}, applied separately to the two coefficient fields, gives
\[
  \eta_K(\mathcal{P}^A)=\eta_K(\mathcal{P}^B)=0.
\]

Case 2.  Suppose instead that $\distT(K,\mathcal S_0)\le L$.  Then there is $K_S\in\mathcal S_0$ with
\[
  \distT(K,K_S)\le L.
\]
For every mesh triangle $T\in\mathcal{T}_h$ with ${\rm int}(T)\subset\omega_L(K)$, we have 
\[
  \distT(T,\mathcal S_0)
  \le \distT(T,K)+\distT(K,K_S)
  \le2L.
\]
Hence $T\in\mathcal N_{2L}(\mathcal S_0)$, and by Property~(M3), the whole patch $\omega_L(K)$ lies in $U_1$.  On $U_1$ the two coefficients $\alpha^A$ and $\alpha^B$ agree by \eqref{eq:alphaB}.  Thus $\mathcal{P}^A|_{\omega_L(K)}=\mathcal{P}^B|_{\omega_L(K)}$, and Definition~\ref{def:local-computable} yields $\eta_K(\mathcal{P}^A)=\eta_K(\mathcal{P}^B)$.
\end{proof}

Therefore, we have shown that the two problem instances produce identical error indicators. Combining this indistinguishability with the contrasting energy estimates from Section~\ref{sec:instances}, we obtain the following central estimate for the product of the reliability and efficiency constants.
\begin{theorem}\label{thm:nogo}
Fix any $L\in\mathbb{N}_0$. Let $\eta$ be a locally computable error estimator with range $L$ satisfying \eqref{eq:rel}--\eqref{eq:eff}, and suppose that its data-oscillation term satisfies \eqref{eq:osc-consistency}. Then
\begin{equation*}
  C_{\rm rel}(M)C_{\rm eff}(M)
  \ge (C_A\sqrt{C_B})^{-1} M^{1/4},\qquad\forall\, M\geq M_0.
\end{equation*}
\end{theorem}

\begin{proof}
By Lemma~\ref{lem:eta-equal},
\[
  \eta(\mathcal{P}^A)=\eta(\mathcal{P}^B).
\]
By \eqref{eq:osc-zero}, the data oscillation vanishes for Problems A and B. It then follows from reliability \eqref{eq:rel} for Problem B, Lemma~\ref{lem:B-lower}, and $u_h^B=0$ that
\begin{equation}\label{eq:uBbound}
    C_B^{-1/2}M^{-1/4}
  \le \|u^B-u_h^B\|_{\alpha^B}
  \le C_{\rm rel}(M)\eta(\mathcal P^B).
\end{equation}
On the other hand, the efficiency bound \eqref{eq:eff} for Problem A, Lemma~\ref{lem:A-upper}, and $u_h^A=0$ give
\begin{equation}\label{eq:uAbound}
\eta(\mathcal P^A)
  \le C_{\rm eff}(M)\|u^A-u_h^A\|_{\alpha^A}
  \le C_A C_{\rm eff}(M)M^{-1/2}.
\end{equation}
Using \eqref{eq:uBbound}, \eqref{eq:uAbound}, and $\eta(\mathcal{P}^A)=\eta(\mathcal{P}^B)$, we obtain
\[
  C_B^{-1/2}M^{-1/4}
  \le C_A C_{\rm rel}(M)C_{\rm eff}(M)M^{-1/2}.
\]
The proof is complete. 
\end{proof}

In particular, there do not exist constants $C_{\rm rel},C_{\rm eff}$ independent of the coefficient contrast $M$ such that both reliability and efficiency hold for all admissible problems.
Since $(1+\log M)^q=o(M^{1/4})$ for every fixed $q\ge0$,
Theorem~\ref{thm:nogo} even rules out estimators that are robust up to any
fixed logarithmic factor in the contrast.

\begin{remark}
There are a few error estimators in the literature accompanied by vertex-patch
\cite{BartelsCarstensen2002,MorinNochettoSiebert2003} or edgewise
\cite{CarstensenVerfurth1999,FeischlPagePraetorius2014} oscillations that
violate the consistency condition \eqref{eq:osc-consistency}. The works
\cite{BartelsCarstensen2002,CarstensenVerfurth1999,FeischlPagePraetorius2014}
concern the Poisson equation, while the results in
\cite{MorinNochettoSiebert2003} for variable-coefficient elliptic equations
are not devoted to contrast-robust a posteriori error estimation.

Using a scaling transformation $\alpha\mapsto a\alpha$, the
same two problem instances can be extended to show that estimators with the vertex/edge-oriented patchwise oscillations above cannot achieve coefficient-contrast robustness. Their natural coefficient-weighted (by the factor $\alpha^{-1/2}$)
analogues can be treated by a minor modification of the proof of
Theorem~\ref{thm:nogo}. Since this extension requires
additional assumptions, we
omit the details to keep the presentation focused.

Some quantitative restriction, such as \eqref{eq:osc-consistency}, is necessary if the additional term in \eqref{eq:rel}--\eqref{eq:eff} is to represent genuine data oscillation: without quantitative control, one could artificially define a “data term” large enough to dominate the entire energy error, making any reliability statement meaningless.
\end{remark}


\bibliographystyle{siamplain}

\end{document}